\documentclass[11pt,a4paper]{article}
\usepackage{graphicx}
\usepackage{amssymb, amsmath, amsthm}
\usepackage[latin1]{inputenc}
\usepackage{hyperref}

\newtheorem{proposition}{Proposition}
\newtheorem{theorem}{Theorem}
\newtheorem{corollary}{Corollary}

\def\R{\mathbb R}

\title{Integral Geometry  about the visual angle of a convex set}
\author{Juli\`a Cuf\'{\i},  Eduardo Gallego,   Agust\'{\i} Revent\'os
\\*[5pt]
{\small Departament de Matem\`{a}tiques}\\ 
{\small Universitat Aut\`{o}noma de Barcelona}\\
{\small 08193 Bellaterra, Barcelona, Catalonia}\\
{\small E-mail: jcufi@mat.uab.cat, egallego@mat.uab.cat, agusti@mat.uab.cat}
}

\date{}

\begin{document} 

\maketitle 

\let\thefootnote\relax
\footnotetext{The authors were  partially supported by grants 2017SGR358, 
2017SGR1725 (Generalitat de Catalunya) and PGC2018-095998-B-100 (Ministerio de Econom\'{\i}a y Competitividad).}

\begin{abstract}
In this paper we deal with a general type of integral formulas  of the visual angle, among them those  of Crofton, Hurwitz and  Masotti,   from the point of view of Integral Geometry. The purpose is twofold: to provide an interpretation of these formulas in terms of integrals of densities  
with respect to the canonical measure in the space of pairs of lines and to give new simpler proofs of them.  \\*[5pt]
\end{abstract}

\section{Introduction}
Throughout this paper $K$ will be a compact convex set in $\R^{2}$ with boundary of class ${\mathcal C}^{2}$. We will denote by $F$ the area of $K$ and by $L$ the length of its boundary. 
 
In  1868 Crofton showed (\cite{crofton}), using arguments that nowadays belong to  Integral Geometry, the well known formula
\begin{equation}\label{21maig-2}
2\int_{P\notin K }(\omega-\sin\omega)\,dP+2\pi F=L^{2}, 
\end{equation}
where 
$\omega=\omega(P)$ is the \emph{visual angle} of $K$ from the point $P$, that is the angle between the two tangents from $P$ to the boundary of $K$. In terms of Integral Geometry 
 both sides of this formula represent the measure of pairs of lines meeting~$K$. In fact the measure of all pairs of lines meeting $K$ is $L^{2}$,  
twice the integral of $\omega-\sin\omega$ with respect to the area element  $dP$
is the measure of those pairs of lines intersecting themselves outside $K$ and  $2\pi F$ is the measure of those  intersecting themselves in~$K$.  

Later on,  Hurwitz in $1902$, in his celebrated paper \cite{Hurwitz1902} on the application of Fourier series to geometric problems, 
considers the integral of some new functions of the visual angle. 
Concretely he proves 
\begin{equation}\label{21maig-3}
\int_{P\notin K}f_{k}(\omega)\,dP=L^{2}+(-1)^{k}\pi^{2}(k^{2}-1)c_{k}^{2}, 
\end{equation}
where 
\begin{equation}\label{maig9-3}
f_{k}(\omega)=-2\sin\omega+\frac{k+1}{k-1}\sin ((k-1)\omega)-\frac{k-1}{k+1}\sin((k+1)\omega), \quad k\geq 2,
\end{equation} 
and $c_{k}^{2}=a_{k}^{2}+b_{k}^{2}$, with $a_{k},b_{k}$  the Fourier coefficients of the support function of~$K$.

In the particular case $k=2$ formula \eqref{21maig-3} gives
\begin{equation}\label{24maig}
\int_{P\notin K}\sin^{3}\omega\,dP=\frac{3}{4}L^{2}+\frac{9}{4}\pi^{2}c_{2}^{2}.
\end{equation}
\medskip

Masotti in $1955$ (\cite{masotti2}) states without proof the following  Crofton's type formula
 \begin{equation}\label{21maig-4}
 \int_{P\notin K}(\omega^{2}-\sin^{2}\omega)\,dP=-\pi^{2}F+\frac{4L^{2}}{\pi}+8\pi\sum_{k\geq 1}\left(\frac{1}{1-4k^{2}}\right)c_{2k}^{2}.
 \end{equation}
In \cite{CGR} a unified approach that encompasses the previous results is provided. As well the following  formula for the integral of any power of the sine function of the visual angle, that generalises \eqref{24maig},  is given:
\begin{multline}\label{21maig-5}
\int_{P\notin K}\sin^{m}(\omega)\,dP=\frac{m!}{2^{m}(m-2)\Gamma(\frac{m-1}{2})^{2}}\, L^{2} \\*[5pt] 
+\frac{m!\pi^{2}}{2^{m-1}(m-2)}\sum_{k\geq 2, \text{ even}}\frac{(-1)^{\frac{k}{2}+1}(k^{2}-1)}{\Gamma(\frac{m+1+k}{2})\Gamma(\frac{m+1-k}{2})}c_{k}^{2}.
\end{multline}

In this paper we deal with a general type of integral formulas of the visual angle including those we have just commented above, from the point of view of Integral Geometry according  to Crofton and Santal\'o \cite{santalo}. The purpose is twofold: to provide an interpretation of these formulas in terms of integrals of densities  
with respect to the canonical measure in the space of pairs of lines and to give new simpler proofs of them.  

\medskip
For each straight line $G$ of the plane that does not pass through the origin let $P$ be the point of $G$ at a minimum distance from the origin. We take as coordinates 
for $G$ the polar coordinates $(p, \varphi)$ of the point $P$, with $p>0$ and $0\leq \varphi <2\pi$.
The invariant measure in the set of lines of the plane not containing the origin
is given by a constant multiple of   
$dG=dp\,d\varphi$.  In the space of ordered pairs of lines   we consider the canonical measure $dG_{1}\, dG_{2}$. This measure  is, except for a constant factor, the only one invariant  under Euclidean motions (see \cite{santalo}).
For every function $\tilde{f}(G_{1},G_{2})$
integrable with respect to  $dG_{1}\, dG_{2}$
we can consider the measure with density $\tilde{f}$, that is  $\tilde{f}(G_{1},G_{2})\,dG_{1}\, dG_{2}$.
We prove in Proposition \ref{maig9b} that this measure is invariant under Euclidean motions if and only if $\tilde{f}(G_{1},G_{2})=f(\varphi_{2}-\varphi_{1})$ with $f$ a $\pi$-periodic function on $\R$.

For such densities, and under some additional hypothesis, it follows from  Theorem \ref{aagg} and Corollary \ref{centredreta} that
\begin{multline}\label{21maig}
A_{0}L^{2}+\pi^{2}\sum_{n\geq 1}c_{2n}^{2}A_{2n}= \int_{G_{i}\cap K\neq\emptyset}f(\varphi_{2}-\varphi_{1})\,dG_{1}\, dG_{2} \\ 
= 2H(\pi) F +2\int_{P\notin K} H(\omega)\,dP,
\end{multline}
where $A_{k}$, $k\geq 0$, are the Fourier coefficients of $f$ corresponding to 
$\cos(k\varphi)$, and $H(x)$ is a ${\mathcal C}^{2}$ function on $[0,\pi]$ satisfying $H''(x)=f(x)\sin (x)$, $x\in [0,\pi]$, and $H(0)=H'(0)=0$. 

The above  
two equalities are the main tools to obtain both new proofs of the formulas discussed above  and their interpretation as integrals of densities with respect to the canonical measure in the space of pairs of lines. As concerning to this  second point, in section \ref{21maig-7} one obtains the following formulas.  

\begin{itemize}
\item[-] {\em Crofton's formula}
$$
\int_{P\notin K}(\omega-\sin\omega)\,dP=-\pi F+\frac{1}{2}\int_{G_{i}\cap K\neq\emptyset}dG_{1}\, dG_{2} .
$$
\end{itemize}
\begin{itemize}
\item[-] {\em Hurwitz's formula}
$$
\int_{P\notin K}f_{k}(\omega)\,dP=\int_{G_{i}\cap K\neq\emptyset}(1+(-1)^{k}(k^{2}-1)\cos(k(\varphi_{2}-\varphi_{1})))\,dG_{1}\, dG_{2}  .
$$
\end{itemize}
\begin{itemize}
\item[-] {\em Masotti's formula}
$$
\int_{P\notin K}(\omega^{2}-\sin^{2}\omega)\,dP=-\pi^{2} F+2\int_{G_{i}\cap K\neq\emptyset}|\sin(\varphi_{2}-\varphi_{1})|\,dG_{1}\, dG_{2} .
$$
\end{itemize}
\begin{itemize}
\item[-] {\em Power sine formula}
\begin{multline*}
\!\!\int_{P\notin K}\sin^{m}\omega\,dP\\*[5pt] 
\!=\frac{1}{2}\int_{G_{i}\cap K\neq\emptyset}\!\!\left(m(m-1)|\mathrm{sin}^{m-3}(\varphi_{2}-\varphi_{1})|\!-\!m^{2}|\mathrm{sin}^{m-1}(\varphi_{2}-\varphi_{1})|\right)\,dG_{1}\, dG_{2}. \!\!\!
\end{multline*}
\end{itemize}
%
%
%

Moreover using  the first equality in \eqref{21maig} one gets the announced  new  proofs of formulas 
\eqref{21maig-2}, \eqref{21maig-4} and \eqref{21maig-5}.

Concerning Hurwitz's integral, when we apply the methods here developed,   it appears a different behavior according to $k$ 
is either even or odd. For $k$ even using~\eqref{21maig} one gets a new proof of \eqref{21maig-3}. Nevertheless when $k$ is odd the density associated 
to the Hurwitz integral  is not $\pi$-periodic since the function $\cos(kx)$ is not, and so we cannot use \eqref{21maig}.    In this case appealing to Proposition \ref{antipi} one obtains   a new result 
that involves a decomposition of the visual angle $\omega$ into two parts $\omega=\omega_{1}+\omega_{2}$
that also have a geometrical interpretation. 

In this setting 
it plays a role the function $f_{k}(\omega)+2(\sin\omega-\omega)$, that is the sum 
of the functions of Hurwitz and Crofton. In spite of $\int_{P\notin K}(f_{k}(\omega)+2(\sin\omega-\omega))\,dP$ depends on $k$, the surprising fact is that, for $k$ odd,  decomposing the visual angle~$\omega$
into the two parts $\omega_{1}$, $\omega_{2}$ 
and adding the corresponding integrals 
leads to 
\begin{equation*}
\int_{P\notin K} \left(f_{k}(\omega_{1})+2(\sin \omega_{1}-\omega_{1})+ f_{k}(\omega_{2})+2(\sin \omega_{2}-\omega_{2})\right)\,dP = 2\pi F,
\end{equation*}
for each $k\geq 3$ odd, as a consequence of Proposition \ref{21maig-6}.

Moreover it will appear that the functions of Crofton and Hurwitz are in some  sense a basis for the integral of any $\pi$-periodic or anti $\pi$-periodic density 
with respect to the measure $dG_{1}\,dG_{2}$ over the set of pairs of lines 
meeting a given compact convex set.

\section{Densities in the space of pairs of lines }
For every function $\tilde{f}(G_{1},G_{2})$ 
defined on the space of pairs of lines  integrable with respect to  the measure $dG_{1}\, dG_{2}$ we consider the measure with density $\tilde{f}$, that is the measure $\tilde{f}(G_{1},G_{2})\,dG_{1}\, dG_{2}$.
The measure of a set~$A$  of pairs of lines in the plane
is then given by 
$$
\int_{A}\tilde{f}(G_{1}, G_{2})\,dG_{1}\, dG_{2}.
$$
We want now to determine when this measure is invariant under Euclidean motions. 
\begin{proposition}\label{maig9b}
The measure $\tilde{f}(G_{1},G_{2})\,dG_{1}\, dG_{2}$
is invariant under the group of Euclidean motions if and only if
$\tilde{f}(G_{1},G_{2})=\tilde{f}(p_{1},\varphi_{1},p_{2},\varphi_{2})=f(\varphi_{2}-\varphi_{1})$
with~$f$ a $\pi$-periodic function on $\R$, where $(p_{i},\varphi_{i})$ are the coordinates of $G_{i}$. 
\end{proposition}

\begin{proof}
The invariance of the measure  is equivalent to the equality $\tilde{f}(p_{1},\!\varphi_{1},p_{2},\!\varphi_{2}) \!=\!\!$ $\tilde{f}(p'_{1},\varphi'_{1},p'_{2},\varphi'_{2})$
for each  Euclidean motion sending the lines with coordinates $(p_{1},\varphi_{1}, p_{2},\varphi_{2})$ to the lines with coordinates  $(p'_{1},\varphi'_{1}, p'_{2},\varphi'_{2})$. 
First of all let us show that $\tilde{f}$
does not depend on $p_{1}$, $p_{2}$.
In fact, for every straight line $G=G(p,\varphi)$ and an arbitrary $a>0$ there is a parallel
line to $G$ with coordinates $(a,\varphi).$ Given two straight lines $G_{1}=G(p_{1},\varphi_{1})$, $G_{2}=G(p_{2},\varphi_{2})$ and two numbers $a_{1},a_{2}>0$ let $G'_{1}$ and $G'_{2}$ be the corresponding parallel lines with coordinates $(a_{1},\varphi_{1})$, $(a_{2},\varphi_{2}).$ Performing the translation that sends the point $G_{1}\cap G_{2}$ to the point $G'_{1}\cap G'_{2}$ we have that $\tilde{f}(p_{1},\varphi_{1},p_{2},\varphi_{2})=\tilde{f}(a_{1},\varphi_{1},a_{2},\varphi_{2})$ and so $\tilde{f}$ does not depend on~$p_{1}$ and~$p_{2}.$

Given now the line $G(p,\varphi)$ if we perform, for instance, the translation given by the vector $-(p+\epsilon)(\cos\varphi,\sin\varphi)$, $\epsilon>0$, the translated line has coordinates $(\epsilon,\varphi+\pi)$. Therefore the function $\tilde{f}$ must be $\pi$-periodic with respect to the arguments $\varphi_{1}$, $\varphi_{2}.$
Finally due to the invariance under rotations  
it follows that $\tilde{f}(p_{1},\varphi_{1},p_{2},\varphi_{2})=\tilde{f}(p_{1},0,p_{2},\varphi_{2}-\varphi_{1})$ and so
$\tilde{f}(p_{1},\varphi_{1},p_{2},\varphi_{2})=f(\varphi_{2}-\varphi_{1})$ with $f$ a $\pi$-periodic function. 
\end{proof}

Our goal is now to integrate a measure given by a density over the set of pairs of lines meeting $K$. In view of Proposition \ref{maig9b} we shall only consider densities which depend on the angle  of the two lines, that is of the form $\tilde{f}(G_{1},G_{2})=f(\varphi_{2}-\varphi_{1})$, with $G_{i}=G_{i}(p_{i},\varphi_{i})$, $i=1,2$. Note that $\varphi_{2}-\varphi_{1}$ gives one of the two angles between the lines $G_{1}$ and $G_{2}$. 

\medskip
We give a formula to compute the integral of the measure $\tilde{f}(G_{1},G_{2})\,dG_{1}\, dG_{2}=f(\varphi_{2}-\varphi_{1})\,dG_{1}\, dG_{2}$, with $f$ a $2\pi$-periodic function  extended to the pairs of lines meeting $K$ in terms of both the Fourier coefficients of $f$ and  of the support function of $K$. Recall that when the origin of coordinates is an interior point of $K$, a hypothesis that we will assume from now on,  the support function $p(\varphi)$
is given by the distance to the origin of the tangent to $K$
whose normal makes and angle~$\varphi$ with  the positive part of the real axis (see \cite{santalo}). 

\begin{theorem}
\label{aagg} Let $K$ be a compact convex set with boundary of length $L$.
Let $f$ be a $2\pi$-periodic continuous function on $\R$ with Fourier expansion
$$
f(\varphi)=\sum_{n\geq 0}A_{n}\cos(n\varphi)+B_{n}\sin(n\varphi).
$$
Then 
\begin{equation}\label{eq31}
\int_{G_{i}\cap K\neq\emptyset}f(\varphi_{2}-\varphi_{1})\,dG_{1}\,dG_{2}=A_{0}L^{2}+\pi^{2}\sum_{n\geq 1}c_{n}^{2}A_{n},
\end{equation}
with  $c_{n}^{2}=a_{n}^{2}+b_{n}^{2}$ where $a_{n}$, $b_{n}$ are the Fourier coefficients of the support 
function~$p(\varphi)$ of  $K$. 
\end{theorem}
\begin{proof}
We have
\begin{equation}\label{14g}
\begin{split}
\int_{G_{i}\cap K\neq\emptyset}\! f(\varphi_{2}-\varphi_{1})\,dG_{1}\,dG_{2}&=\!\int_{0}^{2\pi}\int_{0}^{2\pi}\int_{0}^{p(\varphi_{1})}\int_{0}^{p(\varphi_{2})}\!f(\varphi_{2}-\varphi_{1})\,dp_{1}\,dp_{2}\,d\varphi_{1}\,d\varphi_{2}\\*[5pt]
&=\!\int_{0}^{2\pi}\int_{0}^{2\pi}p(\varphi_{1})p(\varphi_{2})f(\varphi_{2}-\varphi_{1})\,d\varphi_{1}\,d\varphi_{2}.
\end{split}
\end{equation}
Performing the change of variables $\varphi_{2}-\varphi_{1}=w$,  $\varphi_{1}=u$
the integral \eqref{14g} becomes
\begin{equation}\label{14g2}
\int_{0}^{2\pi}p(u)\int_{-u}^{2\pi-u}p(u+w)f(w)\,dw\,du.
\end{equation}
The Fourier development  of $p(u+w)$ in terms of 
the Fourier coefficients $a_{n}$, $b_{n}$ of~$p(u)$  is
\begin{multline*}
p(u+w)=a_{0}+\sum_{n\geq 1}\biggl((a_{n}\cos(nu)+b_{n}\sin(nu))\cos(nw)\\ +(-a_{n}\sin(nu)+b_{n}\cos(nu))\sin(nw)\biggr).
\end{multline*}
By the Plancherel identity the integral \eqref{14g2} 
is equal to
\begin{multline*}
\int_{0}^{2\pi}p(u)\biggl[2\pi \,A_{0}\,a_{0}+\pi\sum_{n\geq 1}A_{n}(a_{n}\cos(nu)+b_{n}\sin(nu))\\*[5pt] 
+B_{n}(-a_{n}\sin(nu)+b_{n}\cos(nu))   \biggr]\,du  \\*[5pt] 
=
\int_{0}^{2\pi}p(u)\biggl[2\pi \,A_{0}\,a_{0}+\pi\sum_{n\geq 1}(A_{n}a_{n}+B_{n}b_{n})\cos(nu)+(A_{n}b_{n}-B_{n}a_{n})\sin(nu))  \biggr]\,du  \\*[5pt] 
=
4\pi^{2}A_{0}\,a_{0}^{2}+\pi^{2}\sum_{n\geq 1}(A_{n}a_{n}+B_{n}b_{n})a_{n}+(A_{n}b_{n}-B_{n}a_{n})b_{n}\\*[5pt]
=
4\pi^{2}A_{0}\,a_{0}^{2}+\pi^{2}\sum_{n\geq 1}A_{n}(a_{n}^{2}+b_{n}^{2})=
A_{0}L^{2}+\pi^{2}\sum_{n\geq 1}A_{n}c_{n}^{2},
\end{multline*}
where we have used that $L=2\pi a_{0}$, which is a consequence of the equality $L=\int_{0}^{2\pi}p(\varphi)\,d\varphi$ (see for instance \cite{santalo}), and the Theorem  is proved. 
\end{proof}

As it is well known (see \cite{Hurwitz1902}) the quantities $c_{k}^{2}=a_{k}^{2}+b_{k}^{2}$, $k\geq 2$, are invariant under Euclidean motions of $K$. 
However $c_{1}^{2}$ changes when moving $K$. So 
the integral in \eqref{eq31} is  invariant under Euclidean motions of $K$ if and only if $A_{1}=0$. In particular this is the case when $f$ is $\pi$-periodic.
\medskip

For a density given by a $\pi$-periodic function $f$ and a compact set of constant width the measure of the pairs of lines that intersect  $K$ is proportional to $L^{2}$. More precisely we have 

\begin{corollary}
Let $K$ be a compact convex set of constant width and $f$ a continuous $\pi$-periodic function. Then
$$
\int_{G_{i}\cap K\neq\emptyset}f(\varphi_{2}-\varphi_{1})\,dG_{1}\,dG_{2}=\lambda L^{2},
$$ 
where $\lambda=(1/\pi)\int_{0}^{\pi}f(\varphi)\,d\varphi$.
\end{corollary}

\begin{proof}
Since $K$ is of constant width the Fourier development of $p(\varphi)$
has only odd terms (see  for instance \S 2 of \cite{CGR}). Moreover the Fourier development of $f$ has only even terms because it is $\pi$-periodic.  Hence \eqref{eq31} gives 
$$
\int_{G_{i}\cap K\neq\emptyset}f(\varphi_{2}-\varphi_{1})\,dG_{1}\,dG_{2}=A_{0}L^{2},
$$ 
with $A_{0}=(1/\pi)\int_{0}^{\pi}f(\varphi)\,d\varphi$. 
\end{proof}

\section{Integral formulas of the visual angle in terms of densities in the space of pairs of lines}
In \cite{CGR} there is a unified approach to several classical formulas involving integrals of functions of the visual angle of a compact convex set $K$. 
Among them one can find the integrals of Crofton, Masotti, powers of  sine,  and Hurwitz.

The original proof of Crofton's formula, via Integral Geometry, involves a measure
on the space of pairs of lines.  The aim of this section is to interpret the formulas in \cite{CGR} in terms of integrals of measures given by densities  
in the space of pairs of lines.

To begin with let us consider Hurwitz's formula
\begin{equation}\label{maig9-2}
\int_{P\notin K}f_{k}(\omega)\,dP=L^{2}+(-1)^{k}\pi^{2}(k^{2}-1)c_{k}^{2}, 
\end{equation}
where $f_{k}(\omega)$ is given in \eqref{maig9-3}. For a proof of \eqref{maig9-2} see \cite{Hurwitz1902} or \cite{CGR}.

Comparing this equality with \eqref{eq31} one gets immediately the following result.

\begin{proposition}\label{maig17}
Let $f_{k}$ be the Hurwitz function defined in \eqref{maig9-3}.
Then 
\begin{equation*}\label{maig9}
\int_{P\notin K}f_{k}(\omega)\,dP=\int_{G_{i}\cap K\neq\emptyset}(1+(-1)^{k}(k^{2}-1)\cos(k(\varphi_{2}-\varphi_{1})))\,dG_{1}\, dG_{2}.
\end{equation*}
\end{proposition}
Nevertheless for the other quoted integral formulas it is not clear at all  what  density must be chosen. We shall provide a general method to find the densities corresponding to integrals  of general functions of  the visual angle.

\subsection{A change of variables}\label{maig7-2}
The classical proof of Crofton's formula is based on the 
change of variables in the space of pairs of lines given by
$$
(p_{1},\varphi_{1},p_{2},\varphi_{2})\longrightarrow (P, \alpha_{1},\alpha_{2}),
$$ 
where $P$ is the intersection point of the two straight lines and $\alpha_{i}\in [0, \pi]$ are the angles which determine the directions of the lines. 
More precisely the angle $\alpha$ associated to a line through a given point $P$ is defined in the following way.  Let $\vec{u}$ be a unitary vector orthogonal  to $\overrightarrow{OP}$  where $O$ is the origin of coordinates, and such that the basis $(\vec{u},\overrightarrow{OP})$ is positively oriented.
Let $G$ be a line through $P$ with unitary director vector $\vec{v}$ such that the basis $(\vec{u},\vec{v})$  is positively oriented. Then $\alpha=\alpha(G)$ is defined by 
$\cos\alpha=\vec{u}\cdot \vec{v}$ and $0<\alpha<\pi$.
From now on we shall say that $\alpha$ is the {\em direction} of the line $G$.

\begin{figure}[h] 
   \centering
   \includegraphics[width=.5\textwidth]{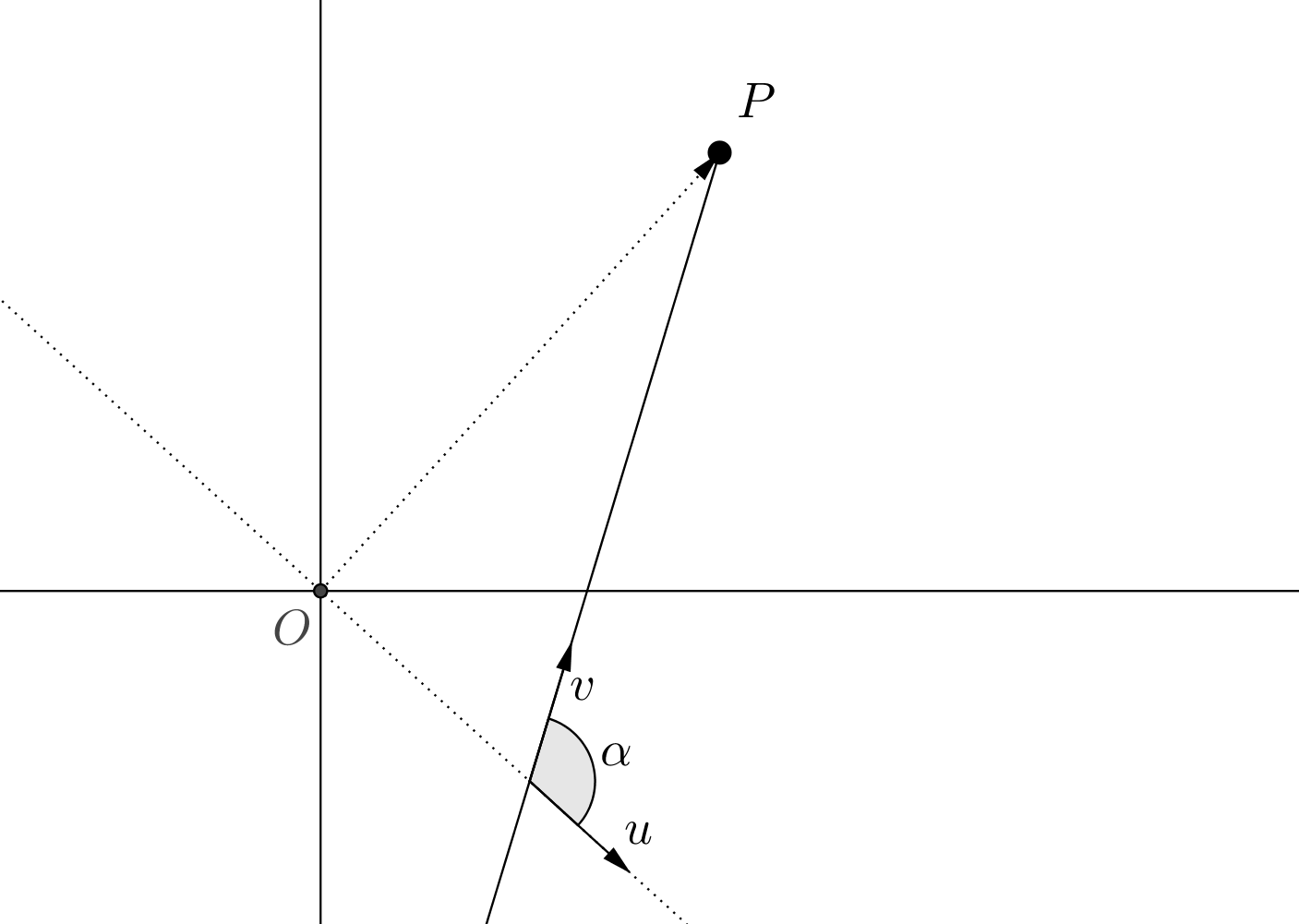}
   \caption{{\em Direction}  of a line.}
   \label{fig:direction}
\end{figure}

\noindent
With these new coordinates, proceeding as in \cite{santalo},  one has 
 \begin{equation}\label{5abr1}
 dG_{1}\, dG_{2}=|\mathrm{sin}(\alpha_{2}-\alpha_{1})|\,d\alpha_{1}\, d\alpha_{2}\, dP.
 \end{equation}
We have used  the fact that $\varphi_{2}-\varphi_{1}=\alpha_{2}-\alpha_{1}+\epsilon\pi$ where $\epsilon= \epsilon(P,\alpha_{1},\alpha_{2})=0,\pm 1$, according to the position with respect to the origin of the lines $G_{1}$, $G_{2}$. As a consequence 
if $f$ is a $\pi$-periodic function we have
\begin{equation}\label{5abril}
f(\varphi_{2}-\varphi_{1})\,dG_{1}\, dG_{2}=f(\alpha_{2}-\alpha_{1})|\mathrm{sin}(\alpha_{2}-\alpha_{1})|\,d\alpha_{1}\, d\alpha_{2}\, dP.
\end{equation}

\enlargethispage{3.5mm}
\vspace*{-9pt}
\subsection{Integrals of functions of pairs of lines meeting a  convex set}\label{maig13} For a point 
$P\notin K$ let $\alpha$, $\beta$ be the directions we have introduced corresponding to the support lines of $K$ through $P$,   with $0<\alpha<\pi/2$ and $\pi/2<\beta<\pi$. 
 Then $\omega =\beta-\alpha$ is the visual angle of $K$ from $P$.
This is the reason why we have slightly modified 
 the definition of the direction angle given by Santal\'o in \cite{santalo}
 as the angle between  the line through $P$ and  the positive $x$ axis, because  with this definition one could have   $\omega =\beta-\alpha$ or $\omega =\pi-(\beta-\alpha)$; see Figure~\ref{fig:omega}.
%
We shall provide now a general formula to calculate the integral of  the right-hand side of \eqref{5abril}.

\begin{proposition}\label{propdreta}
Let $f$ be a $2\pi$-periodic continuous function on $\R$, and $H$  a ${\mathcal C}^{2}$~function on $[-\pi,\pi]$ satisfying the conditions $H''(x)=f(x)\cdot \sin (x)$, $x\in [-\pi,\pi]$, and $ H(0)=H'(0)=0.$
Denote by $\alpha_{i}$  the direction of the line $G_{i}$.  Then
\begin{multline*}\label{25g2}
\int_{G_{i}\cap K\neq\emptyset}f(\alpha_{2}-\alpha_{1})|\mathrm{sin}(\alpha_{2}-\alpha_{1})|\,d\alpha_{1}\,d\alpha_{2}\,dP\\*[-7pt] 
=(H(\pi)-H(-\pi))F+\int_{P\notin K}(H(\omega)-H(-\omega))\,dP,
\end{multline*}
where $\omega=\omega(P)$ is the visual angle of $K$ from $P$.
\end{proposition}

\begin{figure}[ht] 
   \centering
   \includegraphics[width=.5\textwidth]{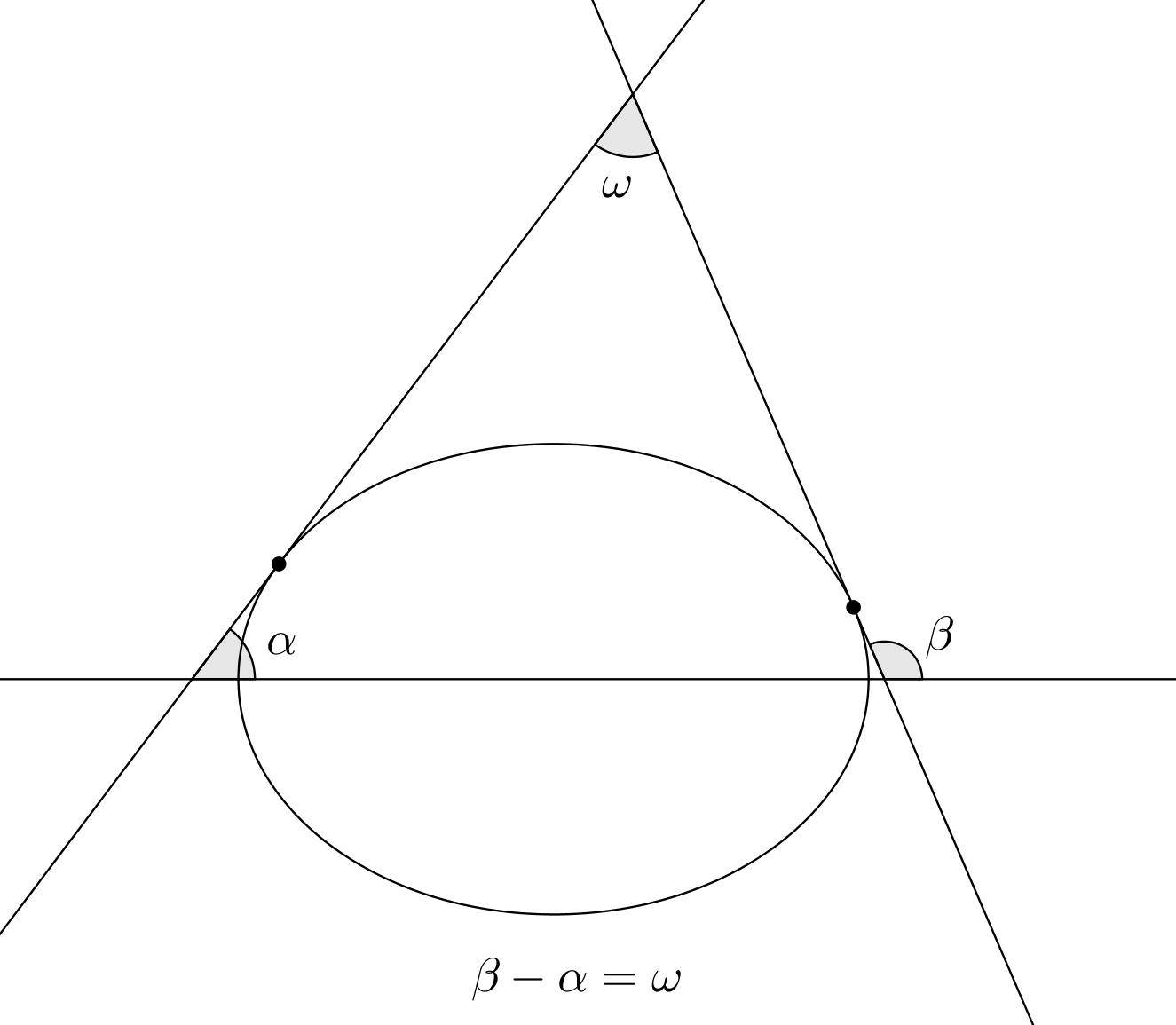}\includegraphics[width=.5\textwidth]{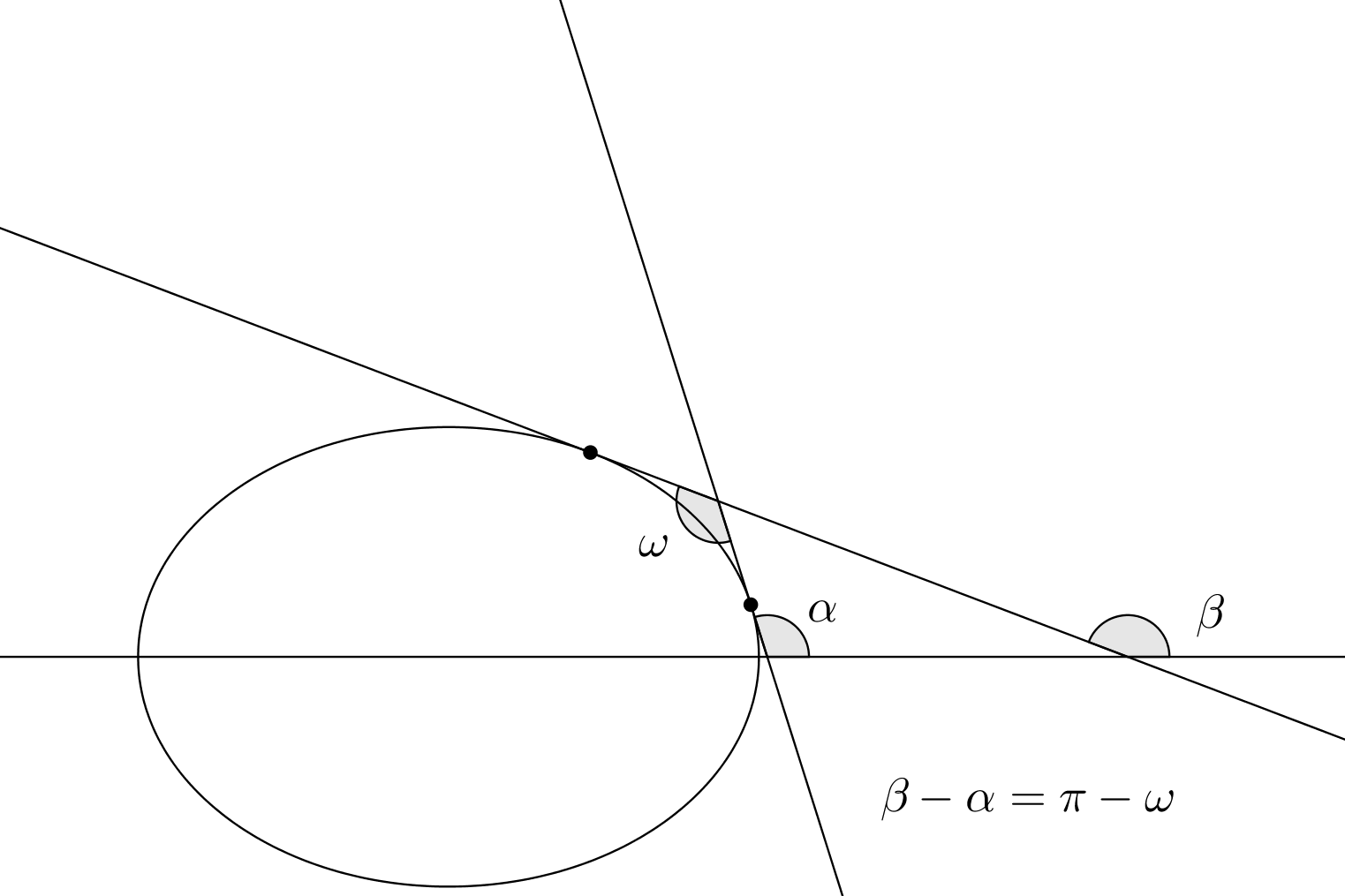}
   \caption{Visual angle of a convex set.}
   \label{fig:omega}
\end{figure}

\begin{proof}
For a given point $P$ in the plane there are angles $\alpha(P)$, $\beta(P)$ such that the pairs of lines $G_{1}$, $G_{2}$ through $P$  that intersect the convex set $K$ are those satisfying  $\alpha(P)\leq \alpha_{i}\leq \beta(P)$, where $\alpha_{i}=\alpha(G_{i})$. When $P\in K$ we have $\alpha(P)=0$ and $\beta(P)=\pi$. 

We need to integrate the function $f(\alpha_{2}-\alpha_{1})|\mathrm{sin}(\alpha_{2}-\alpha_{1})|$ over $[\alpha,\beta]^{2}$ with $\alpha=\alpha(P)$ and $\beta=\beta(P)$. In order to perform this integral we divide $[\alpha,\beta]^{2}$ into the union of the regions $\mathcal{R}_{1}=\{(\alpha_{1},\alpha_{2})\in [\alpha,\beta]^{2}: \alpha_{2}\geq \alpha_{1}\}$ and $\mathcal{R}_{2}=\{(\alpha_{1},\alpha_{2})\in [\alpha,\beta]^{2}: \alpha_{2}< \alpha_{1}\}.$ Therefore
\begin{multline*}
\int_{[\alpha,\beta]^{2}}f(\alpha_{2}-\alpha_{1})|\mathrm{sin}(\alpha_{2}-\alpha_{1})|\,d\alpha_{1} \,d\alpha_{2}\\
=\int_{\mathcal{R}_{1}}f(\alpha_{2}-\alpha_{1})\sin(\alpha_{2}-\alpha_{1})\,d\alpha_{1} \,d\alpha_{2}-\int_{\mathcal{R}_{2}}f(\alpha_{2}-\alpha_{1})\sin(\alpha_{2}-\alpha_{1})\,d\alpha_{1} \,d\alpha_{2}\\
=\int_{\alpha}^{\beta}\left(\int_{\alpha}^{\alpha_{2}}f(\alpha_{2}-\alpha_{1})\sin(\alpha_{2}-\alpha_{1})\,d\alpha_{1}\right)\,d\alpha_{2}\\
-\int_{\alpha}^{\beta}\left(\int_{\alpha}^{\alpha_{1}}f(\alpha_{2}-\alpha_{1})\sin(\alpha_{2}-\alpha_{1})\,d\alpha_{2}\right)\,d\alpha_{1}\\
=\int_{\alpha}^{\beta}\left[-H'(\alpha_{2}-\alpha_{1})\right]_{\alpha}^{\alpha_{2}}\,d\alpha_{2}-\int_{\alpha}^{\beta}\left[H'(\alpha_{2}-\alpha_{1})\right]_{\alpha}^{\alpha_{1}}\,d\alpha_{1}
\\
=\left[H(\alpha_{2}-\alpha)\right]_{\alpha}^{\beta}-\left[H(\alpha-\alpha_{1})\right]_{\alpha}^{\beta}
=H(\beta-\alpha)-H(\alpha-\beta).
\end{multline*}
Hence
\begin{multline*}
\int_{G_{i}\cap K\neq\emptyset}f(\alpha_{2}-\alpha_{1})|\mathrm{sin}(\alpha_{2}-\alpha_{1})|\,d\alpha_{1}\,d\alpha_{2}\,dP \\ =\biggl(\int_{P\in K}+\int_{P\notin K}\biggr) (H(\beta-\alpha)-H(\alpha-\beta))\,dP.
\end{multline*}
Taking into account that the visual angle $\omega(P)$ is given by $\beta(P)-\alpha(P)$ the result follows. 
\end{proof}

In the next result we assume the additional hypothesis that $f(x)$ is an even function. 

\begin{proposition}\label{prop54}
 Let $f$ be a $2\pi$-periodic continuous function on $\R$, with $f(-x)=f(x)$, $x\in\R$,  and $H$  a ${\mathcal C}^{2}$ function on $[0,\pi]$ satisfying the conditions $H''(x)=f(x)\cdot \sin (x)$, $x\in [0,\pi]$, and $ H(0)=H'(0)=0.$
Denote by $\alpha_{i}$  the direction of the line $G_{i}$.  Then
\begin{equation*}\label{25g}
\int_{G_{i}\cap K\neq\emptyset}f(\alpha_{2}-\alpha_{1})|\mathrm{sin}(\alpha_{2}-\alpha_{1})|\,d\alpha_{1}\,d\alpha_{2}\,dP=2H(\pi) F +2\int_{P\notin K} H(\omega)\,dP,
\end{equation*}
where $\omega=\omega(P)$ is the visual angle of $K$ from $P$.
\end{proposition}

\begin{proof} 
Just proceed as in the above proof taking into account that 
\begin{multline*}
\int_{[\alpha,\beta]^{2}}f(\alpha_{2}-\alpha_{1})|\mathrm{sin}(\alpha_{2}-\alpha_{1})|\,d\alpha_{1} \,d\alpha_{2}\\ =2\int_{\alpha}^{\beta}\left(\int_{\alpha}^{\alpha_{2}}f(\alpha_{2}-\alpha_{1})\sin(\alpha_{2}-\alpha_{1})\,d\alpha_{1}\right)\,d\alpha_{2}.
\end{multline*}
\end{proof}

For the special case where $f$ is a $\pi$-periodic function one has
\begin{corollary}\label{centredreta}  
Let $f$ be a $\pi$-periodic continuous function on $\R$, and $H$  a ${\mathcal C}^{2}$ function on $[-\pi,\pi]$ satisfying the conditions $H''(x)=f(x)\cdot \sin (x)$, $x\in [-\pi,\pi]$, and $ H(0)=H'(0)=0.$ Then 
\begin{equation*}
\label{25g4}
\int_{G_{i}\cap K\neq\emptyset}f(\varphi_{2}-\varphi_{1})\,dG_{1}\,dG_{2}\!=\! \left((H(\pi)-H(-\pi))F+\!\int_{P\notin K}(H(\omega)-H(-\omega))\,dP\right).\!
\end{equation*}
 If moreover $f(-x)=f(x)$ and $H(x)$ is ${\mathcal C}^{2}$ on $[0,\pi]$ with  $H''(x)=f(x)\cdot \sin (x)$, $x\in [0,\pi]$, and $ H(0)=H'(0)=0,$ one has
 \begin{equation}\label{25gg}
\int_{G_{i}\cap K\neq\emptyset}f(\varphi_{2}-\varphi_{1})\,dG_{1}\,dG_{2}=2H(\pi) F +2\int_{P\notin K} H(\omega)\,dP.
\end{equation}
\end{corollary}

\begin{proof}
When $f$ is a $\pi$
-periodic function we have equality \eqref{5abril} and the result is then a consequence of Proposition  \eqref{propdreta} and Proposition \eqref{prop54}. 
\end{proof}

Integral formulas as those given in \eqref{eq31} and \eqref{25gg} open the possibility to prove interesting relations for quantities linked to convex sets. For instance when applied to the function $f(x)=\cos kx$ they give Hurwitz's formula \eqref{maig9-2} for $k$ even (see section \ref{hurwitzsec}). 

For odd values of $k$ the Corollary \ref{centredreta} does not apply  because $f(x)=\cos kx$ is not a $\pi$-periodic function. In this case we have $f(x+\pi)=-f(x)$ and we say that $f$ is an \emph{anti $\pi$-periodic} function.
For this type of functions we can modify the above proofs to obtain a new result that involve a decomposition of the visual angle $\omega$ into $\omega=\omega_{1}+\omega_{2}$ where $\omega_{1}$ and $\omega_{2}$ are defined in the following way.
 Given a point $P\notin K$ we have considered in section \ref{maig13} the directions $0<\alpha<\pi/2<\beta<\pi$  of the support lines of $K$ through $P$ and the visual angle $\omega=\beta-\alpha$. Let us take $\omega_{1}=\pi/2-\alpha$ and $\omega_{2}=\beta-\pi/2$. Then we have
\begin{figure}[h] 
   \centering
   \includegraphics[width=.4\textwidth]{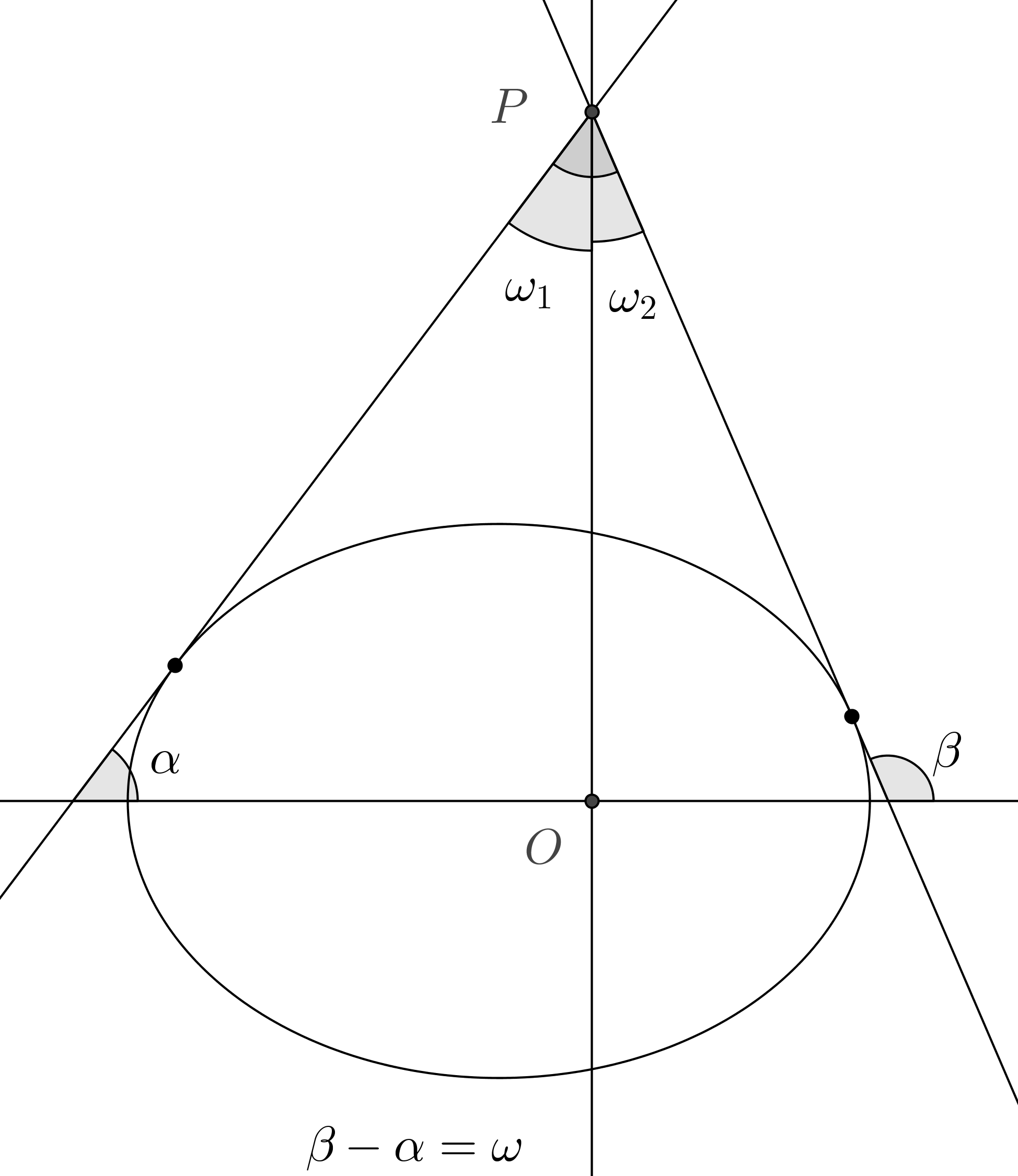}
   \caption{Angles $\omega_{1}$ and $\omega_{2}$.}
   \label{fig:omega12}
\end{figure}
\begin{proposition}\label{antipi} Let $f$ be an anti $\pi$-periodic continuous function on $\R$ such that $f(x)=f(-x)$ and $H$  a ${\mathcal C}^{2}$ function on $[0,\pi]$ with  $H''(x)=f(x)\cdot \sin (x)$, $x\in [0,\pi]$, and $ H(0)=H'(0)=0.$  Then 
\begin{multline}\label{formanti}
\int_{G_{i}\cap K\neq\emptyset}f(\varphi_{2}-\varphi_{1})\,dG_{1}\,dG_{2}\\ 
=2(2H(\pi/2)-H(\pi)) F +2\int_{P\notin K} (2H(\omega_{1})+2H(\omega_{2})-H(\omega))\,dP.
\end{multline}
\end{proposition}
\begin{proof}
In section \ref{maig7-2} we have seen that $\varphi_{2}-\varphi_{1}=\alpha_{2}-\alpha_{1}+\epsilon\pi$ where $\epsilon= \epsilon(P,\alpha_{1},\alpha_{2})=0,\pm 1$. Then
\begin{multline*}
\int_{G_{i}\cap K\neq\emptyset}f(\varphi_{2}-\varphi_{1})\,dG_{1}\,dG_{2}\\
=\int_{P\in\R^{2}}\int_{[\alpha(P),\beta(P)]^{2}}(-1)^{\epsilon}f(\alpha_{2}-\alpha_{1})|\mathrm{sin}(\alpha_{2}-\alpha_{1})|\,d\alpha_{1}\, d\alpha_{2}\, dP.
\end{multline*} 
If $P\notin K$  we consider the regions
\begin{align*}
\mathcal{R}_{1}&=\{(\alpha_{1},\alpha_{2}): \alpha\leq \alpha_{1}<\alpha_{2}\leq \pi/2\},\\
\mathcal{R}_{2}&=\{(\alpha_{1},\alpha_{2}): \pi/2\leq \alpha_{1}<\alpha_{2}\leq \beta\},\\
\mathcal{R}_{3}&=\{(\alpha_{1},\alpha_{2}): \alpha\leq \alpha_{1}<\pi/2<\alpha_{2}\leq \beta\}.
\end{align*}
In  $\mathcal{R}_{1}$ and $\mathcal{R}_{2}$ we have $\epsilon=1$ and $\epsilon=-1$ in region $\mathcal{R}_{3}$. Therefore, for $P\notin K$
\begin{multline*}
\int_{\alpha}^{\beta}\int_{\alpha}^{\beta}(-1)^{\epsilon}f(\alpha_{2}-\alpha_{1})|\mathrm{sin}(\alpha_{2}-\alpha_{1})|\,d\alpha_{1}\, d\alpha_{2}\\*[5pt]
=2\left(\int_{\mathcal{R}_{1}}f(\alpha_{2}-\alpha_{1})\sin(\alpha_{2}-\alpha_{1})\,d\alpha_{1}\, d\alpha_{2}+
\int_{\mathcal{R}_{2}}f(\alpha_{2}-\alpha_{1})\sin(\alpha_{2}-\alpha_{1})\,d\alpha_{1}\, d\alpha_{2}\right.\\*[5pt]
=\left. -\int_{\mathcal{R}_{3}}f(\alpha_{2}-\alpha_{1})\sin(\alpha_{2}-\alpha_{1})\,d\alpha_{1}\, d\alpha_{2}\right).
\end{multline*}
The integrals over $\mathcal{R}_{1}$ and $\mathcal{R}_{2}$ are easily computed and their values are $2H(\omega_{1})$ and $2H(\omega_{2})$ respectively. Let us compute the third integral.
\begin{multline*}
\int_{\mathcal{R}_{3}}f(\alpha_{2}-\alpha_{1})\sin(\alpha_{2}-\alpha_{1})\,d\alpha_{1}\, d\alpha_{2} \\*[5pt] 
=\int_{\pi/2}^{\beta}\int_{\alpha}^{\pi/2}f(\alpha_{2}-\alpha_{1})\sin(\alpha_{2}-\alpha_{1})\,d\alpha_{1}\, d\alpha_{2})\\*[5pt]
=2\int_{\pi/2}^{{\beta}}\left[-H'(\alpha_{2}-\alpha_{1})\right]_{\alpha_{1}=\alpha}^{{\alpha_{1}=\pi/2}}\,d\alpha_{2}=
2\int_{\pi/2}^{\beta}(H'(\alpha_{2}-\alpha)-H'(\alpha_{2}-\pi/2))\,d\alpha_{2}\\*[5pt]
=2\left[H(\alpha_{2}-\alpha)-H(\alpha_{2}-\pi/2)\right]_{\alpha_{2}=\pi/2}^{\alpha_{2}=\beta}=2\left(H(\omega)-H(\beta-\pi/2)-H(\pi/2-\alpha)\right)\\*[5pt]
=2\left(H(\omega)-H(\omega_{2})-H(\omega_{1})\right).
\end{multline*}
Finally, for $G_{1}\cap G_{2}=P\notin K$ we have
\begin{equation}\label{pnotinantip}
\int_{P\notin K}f(\varphi_{2}-\varphi_{1})\,dG_{1}\,dG_{2}=2\int_{P\notin K}(2H(\omega_{1})+2H(\omega_{2})-H(\omega))\, dP.
\end{equation}
When $P\in K$ we do the same computations but now $\alpha=0$, $ \beta=\pi$ and $\omega=\beta-\alpha=\pi$ and so $\omega_{1}=\pi/2=\omega_{2}$. Thus 
\begin{equation}\label{pinantip}
\int_{G_{1}\cap G_{2}\in K}f(\varphi_{2}-\varphi_{1})\,dG_{1}\,dG_{2}=2(4H(\pi/2)-H(\pi))F.
\end{equation}
Joining \eqref{pnotinantip} and \eqref{pinantip} the Proposition follows. 
\end{proof}

\subsection{Interpretation in terms of densities of the formulas of Crofton, Masotti and powers of sine}\label{21maig-7}
In this section we will give an interpretation of the integrals of the visual angle appearing in the formulas of Crofton, Masotti and power of sine in terms of integrals of densities in the space of pairs of lines. For Hurwitz's formula this was done in Proposition \ref{maig17}. 

\subsubsection*{Crofton's formula}  
Taking $H(x)=x-\sin(x)$ it follows that $f=1$ in Corollary \ref{centredreta} and since $H(\pi)=\pi$ using \eqref{25gg} we get

\begin{proposition}\label{prop45} The following equality holds.
$$
\int_{G_{i}\cap K\neq \emptyset}\,dG_{1}\,dG_{2}=2\pi F+2\int_{P\notin K}(\omega-\sin\omega)\,dP.
$$
\end{proposition}

\subsubsection*{Masotti's formula}  
Taking $H(x)=x^{2}-\sin^{2}(x)$ one gets $H''(x)/\sin(x)=4\sin(x)$. So the function $f(x)=4|\mathrm{sin}(x)|$, $x\in \R$, satisfies  the hypothesis of Corollary \ref{centredreta} and equation \eqref{25gg} gives 
\begin{proposition}\label{prop46} The following equality holds

\begin{equation*}
2\int_{G_{i}\cap K\neq\emptyset}|\mathrm{sin}(\varphi_{2}-\varphi_{1})|\,dG_{1}\,dG_{2}=\pi^{2} F +\int_{P\notin K} (\omega^{2}-\sin^{2}\omega)\,dP.
\end{equation*}
\end{proposition}

\subsubsection*{Powers of sine formula}   
Finally, in an analogous way we can interpretate   the integral of any power of the sine of the visual angle.  Effectively for $H(x)= \sin^m(x)$ it follows that 
 $$
 H''(x)/\sin(x)=m(m-1)\sin^{m-3}(x)-m^{2}\sin^{m-1}(x).
 $$
 So taking  $f(x)=m(m-1)|\mathrm{sin}^{n-3}(x)|-m^{2}|\mathrm{sin}^{m-1}(x)|$ 
 the hypothesis of Corollary~\ref{centredreta} are satisfied and by \eqref{25gg}  we have
\begin{proposition}\label{prop47}The following equality holds
\begin{multline*}
2\int_{P\notin K}\sin^{m}(\omega)\,dP\\*[5pt]
=\int_{G_{i}\cap K\neq\emptyset}\left(m(m-1)|\mathrm{sin}^{m-3}(\varphi_{2}-\varphi_{1})|-m^{2}|\mathrm{sin}^{m-1}(\varphi_{2}-\varphi_{1})|\right)\,dG_{1}\,dG_{2}.
\end{multline*}
\end{proposition}
\section{New proofs of classical formulas}\label{seccio4}
Combining the results of the previous section with Theorem \ref{aagg}  new proofs of the formulas of Masotti and the powers of sine can be obtained, in the spirit of the classical proof of Crofton's formula via Integral Geometry

To begin with we note that Theorem \ref{aagg} implies 
the equality  $\int_{G_{i}\cap K\neq \emptyset}\,dG_{1}\,dG_{2}=L^{2}$ which is also
an immediate consequence of the well known  Cauchy--Crofton's formula (see \cite{santalo}). Now this equality together   
with  Proposition \ref{prop45} gives Crofton's formula  
\begin{equation}\label{crofton18}
L^{2}=2\pi F+2\int_{P\notin K}(\omega-\sin\omega)\,dP.
\end{equation}
  
\subsubsection*{Masotti's formula}  
A simple calculation shows that  the Fourier expansion 
of the function $|\mathrm{sin}(t)|$ is
\begin{equation}\label{8abril}
|\mathrm{sin}(t)|=\frac{2}{\pi}+\frac{4}{\pi}\sum_{n\geq 1}\frac{\cos(2nt)}{1-4n^{2}}.
\end{equation}
So by Theorem \ref{aagg},
$$
\int_{G_{i}\cap K\neq\emptyset} |\mathrm{sin}(\varphi_{2}-\varphi_{1})|\,dG_{1}\,dG_{2}=\frac{2L^{2}}{\pi}+4\pi\sum_{n\geq 1}\frac{c_{2n}^{2}}{1-4n^{2}},
$$
and using Proposition \ref{prop46}
one gets 
\begin{equation*}\label{20jb}
\int_{P\notin K}(\omega^{2}-\sin^{2}\omega)\,dP=-\pi^{2}F+\frac{4L^{2}}{\pi}+8\pi\sum_{n\geq 1}\frac{c_{2n}^{2}}{1-4n^{2}},
\end{equation*}
which is Masotti's formula \eqref{21maig-4}.  

\subsubsection*{Another example}  
In the preceding sections we have interpreted integral formulas of some functions of the visual angle in terms of densities in the space of pairs of lines.
But one can also proceed in the reverse sense, that is to start from a density and to look for the corresponding  function of the visual angle.

For instance the proof of Masotti's formula leads to compute  
$\int_{G_{i}\cap K}|\mathrm{sin}(\varphi_{2}-\varphi_{1})|\,dG_{1}dG_{2}$.
If we consider now the density function $|\mathrm{cos}(\varphi_{2}-\varphi_{1})|$, using Theorem~\ref{aagg} and that 
\begin{equation*}
|\mathrm{cos}(t)|=\frac{2}{\pi}+\frac{4}{\pi}\sum_{n\geq 1}\frac{(-1)^{n}\cos(2nt)}{1-4n^{2}}
\end{equation*}
 we get
$$
\int_{G_{i}\cap K\neq\emptyset} |\mathrm{cos}(\varphi_{2}-\varphi_{1})|\,dG_{1}\,dG_{2}=\frac{2L^{2}}{\pi}+4\pi\sum_{n\geq 1}\frac{(-1)^{n}c_{2n}^{2}}{1-4n^{2}}.
$$
The function $H$ appearing in  Corollary \ref{centredreta}  
is in this case
$$
H(\omega)=
\begin{cases}
\frac{1}{4}(\omega-\sin\omega\cos\omega)& 0\leq\omega\leq \pi/2,\\[.3cm]
\frac{1}{4}(3\omega-\pi+\sin\omega\cos\omega)&\pi/2\leq \omega\leq \pi.
\end{cases}
$$
Hence, by \eqref{25gg} we have
$$
\int_{G_{i}\cap K\neq\emptyset} |\mathrm{cos}(\varphi_{2}-\varphi_{1})|\,dG_{1}\,dG_{2}=\pi F+ 2\int_{P\notin K}H(\omega)\,dP.
$$

\subsubsection*{Powers of sine formula}  
In order to apply Theorem \ref{aagg} to the right-hand side of the equality in  Proposition~\ref{prop47} we need to compute 
the Fourier coefficients of the function $f(x)=m(m-1)|\mathrm{sin}^{m-3}(x)|-m^{2}|\mathrm{sin}^{m-1}(x)|$. 
It is clear that $A_{k}=0$ for $k$ odd. For $k$ even  we have
\begin{multline}\label{maig}
A_{k}=\frac{1}{\pi}\int_{0}^{2\pi}f(x)\cos(kx)\,dx\\*[5pt]
=\frac{1}{\pi}\left[2m(m-1)\int_{0}^{\pi}\sin^{m-3}x\cos(kx)\,dx
-2m^{2}\int_{0}^{\pi}\sin^{m-1}x\cos(kx)\,dx\right]\\*[5pt]
=\frac{1}{\pi}[2m(m-1)I_{m-3,k}-2m^{2}I_{m-1,k}],
\end{multline}
where
$$
I_{m,k}=\int_{0}^{\pi}\sin^{m}(x)\cos(kx)\,dx=(-1)^{k/2}\frac{2^{-m}m!\pi}{\Gamma(1+\frac{m-k}{2})\Gamma(1+\frac{m+k}{2})},
$$
(see, for instance, \cite{grads}, p. 372).
Substituting this expression in \eqref{maig} it follows
$$
A_{k}=\frac{m!}{2^{m-2}(m-2)}\frac{(-1)^{\frac{k}{2}+1}(k^{2}-1)}{\Gamma(\frac{m+1+k}{2})\Gamma(\frac{m+1-k}{2})}.
$$ 
Finally using Theorem \ref{aagg} we get
\begin{multline*}\label{21maig-5}
\int_{P\notin K}\sin^{m}(\omega)\,dP=\frac{m!}{2^{m}(m-2)\Gamma(\frac{m-1}{2})^{2}}\, L^{2} \\ 
+\frac{m!\pi^{2}}{2^{m-1}(m-2)}\sum_{k\geq 2, even}\frac{(-1)^{\frac{k}{2}+1}(k^{2}-1)}{\Gamma(\frac{m+1+k}{2})\Gamma(\frac{m+1-k}{2})}c_{k}^{2}.
\end{multline*}
Note that for $m$ odd the index $k$ in the sum runs only from $2$ to $m-1$.

This formula, which was first obtained by a different method in \cite{CGR}, provides an interpretation of the coefficients of $c_{k}^{2}$ as the Fourier coefficients of the above function $f$.

\subsubsection*{Crofton-Hurwitz's integral} \label{hurwitzsec}
In the above two previous sections
we have strongly used 
equality \eqref{25gg} of Corollary~\ref{centredreta} that depends  
on the fact that the function $f(x)$ is $\pi$-periodic,  a fact that is crucial  in order that equality \eqref{5abril} holds.

Consider now the function  $f(x)=\cos kx$ with $k>1$.  This function satisfies the hypothesis of Corollary~\ref{centredreta} for $k$ even and the hypothesis of Proposition~\ref{antipi} for $k$~odd. We have that
\begin{equation}\label{Hhurwitzk}
H_k(x)=\frac{1}{2(k^{2}-1)}\left(f_{k}(x)+2(\sin x-x)\right),
\end{equation}
with $f_{k}(x)$ the Hurwitz's function given in \eqref{maig9-3}, satisfies the equation $H_k''(x)=\cos kx\cdot \sin x,$ $x\in [0,\pi],$ and $H_k(0)=H_k'(0)=0.$ 
Therefore, for $k$ even, equalities~\eqref{eq31} and~\eqref{25gg} give
$$
\pi^{2}c_{k}^{2}= \int_{G_{i}\cap K\neq\emptyset}\cos(k(\varphi_{2}-\varphi_{1}))\,dG_{1}\,dG_{2} =-\frac{\pi F}{k^{2}-1}+2\int_{P\notin K}H_{k}(\omega)\,dP,
$$
and using Crofton's formula \eqref{crofton18} one gets a new proof of Hurwitz's formula \eqref{maig9-2} for $k$ even. 

When $k$ is odd equation  $\eqref{formanti}$ gives
\begin{multline*}
\int_{G_{i}\cap K\neq\emptyset}\cos k(\varphi_{2}-\varphi_{1})\,dG_{1}\,dG_{2}=\\ =-{2\pi F\over k^{2}-1}+
2\int_{P\notin K} (2H_k(\omega_{1})+2H_k(\omega_{2})-H_k(\omega))\ dP.
\end{multline*}
Using the equality  \eqref{eq31} one deduces that
$$
\int_{P\notin K}H_k(\omega)\,dP=-{\pi^{2}c_{k}^{2}\over 2}-{2\pi F\over k^{2}-1}+\int_{P\notin K} (H_k(\omega_{1})+H_k(\omega_{2})\,dP.
$$
Now by \eqref{Hhurwitzk} and  Crofton's formula  we obtain
\begin{equation}\label{hwtzsenar}
\int_{P\notin K}f_{k}(\omega)\,dP=L^{2}-\pi^{2}(k^{2}-1)c_{k}^{2}-2\pi F+2(k^{2}-1)\int_{P\notin K} (H_k(\omega_{1})+H_k(\omega_{2}))\,dP.
\end{equation}
Since we do not know the value of $\int_{P\notin K} (H_k(\omega_{1})+H_k(\omega_{2}))\,dP$ we are not able to prove Hurwitz formula in the case of $k$ odd.  But from \eqref{maig9-2}   we get the following result. 

\begin{proposition}\label{21maig-6} Let $K$ be  a compact convex set of area $F$.  Then
\begin{equation}\label{maig13-2}
(k^{2}-1)\int_{P\notin K} (H_k(\omega_{1})+H_k(\omega_{2})\,dP=\pi F
\end{equation}
 for each $k\geq 3$ odd, where $H_{k}$ is given in \eqref{Hhurwitzk}.
\end{proposition}
Notice that the above equation is equivalent to  
\begin{equation}\label{maig15}
\int_{P\notin K} \left(f_{k}(\omega_{1})+2(\sin \omega_{1}-\omega_{1})+ f_{k}(\omega_{2})+2(\sin \omega_{2}-\omega_{2})\right)\,dP = 2\pi F.
\end{equation}

The function $H_{k}$ is the sum, except for a constant, of Hurwitz's function and Crofton's function and so are the terms in the above integrand.  
The integral of the sum of Crofton's
and Hurwitz's functions of the visual angle is  
$$
\int_{P\notin K}(f_{k}(\omega)+2(\sin \omega-\omega))\,dP= 2\pi F+ (-1)^{k}\pi^{2}(k^{2}-1)c_{k}^{2}, \quad k\geq 2.
$$
The surprising fact is that, for $k$ odd,  decomposing the visual angle $\omega$
into the two parts $\omega=\omega_{1}+\omega_{2}$ 
and adding the corresponding integrals 
one gets \eqref{maig15}
in which the right-hand side does not depend on $k$.
\medskip	

In concluding we make the following remark. Theorem \ref{aagg} states that the integral  $\int_{G_{i}\cap K\neq\emptyset}f(\varphi_{2}-\varphi_{1})\,dG_{1}\,dG_{2}$ depends only on the integrals $\int_{G_{i}\cap K\neq\emptyset}\cos k(\varphi_{2}-\varphi_{1})\,dG_{1}\,dG_{2}$. So, by the results of section \ref{maig7-2} we are lead to calculate the functions~$H_{k}(x)$ such that $H_{k}''(x)=\cos(kx)\sin(x)$ with $H_{k}(0)=H_{k}'(0)=0$. These functions appear to be the sum of the functions of Hurtwitz and Crofton given in~\eqref{Hhurwitzk}, that is 
$$
H_k(x)=\frac{1}{2(k^{2}-1)}\left(f_{k}(x)+2(\sin x-x)\right),\quad k\geq 2,
$$
and
$H_{1}(x)=(1/8)(2x-\sin(2x)).$ 

As a consequence  when $f$ is a $\pi$-periodic density,  according to Corollary \ref{centredreta}, 
the integral $\int_{G_{i}\cap K\neq\emptyset}f(\varphi_{2}-\varphi_{1})\,dG_{1}\,dG_{2}$ is a linear combination of integrals extended outside $K$
of the functions of the visual angle $H_{k}(\omega)$.
Likewise when the density~$f$ is anti $\pi$-periodic, according to Proposition 
 \ref{antipi}, the corresponding integral of the density 
is
a linear combination of integrals extended outside $K$
of the functions~$H_{k}(\omega)$, $H_{k}(\omega_{1})$ and $H_{k}(\omega_{2})$. 
\medskip	

Summarizing, it appears that the functions of Crofton and Hurwitz are some  kind of basis for the integral of any $\pi$-periodic or anti $\pi$-periodic density 
with respect to the measure $dG_{1}\,dG_{2}$ over the set of pairs of lines 
meeting a given compact convex set. 
\bibliographystyle{plain}
\bibliography{CGRdP}

\begin{thebibliography}{1}

\bibitem{crofton}
M.~W. Crofton.
\newblock On the theory of local probability.
\newblock {\em Phil. Trans. R. Soc. Lond.}, 158:181--199, 1868.

\bibitem{CGR}
J.~Cuf\'\i, E.~Gallego, and A.~Revent\'os.
\newblock {On the integral formulas of Crofton and Hurwitz relative to the
  visual angle of a convex set}.
\newblock {\em Mathematika}, 65:874--896, 2019.

\bibitem{grads}
I.~S. Gradshteyn and I.~M. Ryzhik.
\newblock {\em Table of integrals, series, and products}.
\newblock Academic Press, New York-London-Toronto, Ont., 1980.

\bibitem{Hurwitz1902}
A.~Hurwitz.
\newblock {Sur quelques applications geometriques des s{\'{e}}ries de Fourier}.
\newblock {\em Annales scientifiques de l'{\'{E}}.N.S., 3{\`{e}}me
  s{\'{e}}rie}, 19:357--408, dec 1902.

\bibitem{masotti2}
G.~Masotti.
\newblock La {G}eometria {I}ntegrale.
\newblock {\em Rend. Sem. Mat. Fis. Milano}, 25:164--231 (1955), 1953--54.

\bibitem{santalo}
L.~A. Santal\'o.
\newblock {\em Integral geometry and geometric probability}.
\newblock Cambridge University Press, Cambridge, second edition, 2004.

\end{thebibliography}

\end{document}